\documentclass{article}
\usepackage{amsmath,amsthm,amssymb}

\usepackage[colorlinks=true,linkcolor=blue,citecolor=blue,pdfpagelabels=false]{hyperref}

\usepackage{thmtools}
\theoremstyle{plain}
  \declaretheorem[numberwithin=section]{theorem}
  \declaretheorem[numberlike=theorem]{corollary}
  \declaretheorem[numberlike=theorem]{proposition}
  \declaretheorem[numberlike=theorem]{lemma}

\theoremstyle{definition}
  
  \declaretheorem[numberlike=theorem]{example}
  
\newenvironment{acknowledgements}{\bigskip\textbf{Acknowledgements.}}{}

\newcommand{\md}{\mathrm{d}}
\newcommand{\email}[1]{{\textit{Email:} \texttt{#1}}}

\begin{document}

\title{Supercongruences for polynomial analogs of the Ap\'ery numbers}

\author{Armin Straub\thanks{\email{straub@southalabama.edu}}\\
Department of Mathematics and Statistics\\
University of South Alabama}

\date{March 19, 2018}

\maketitle

\begin{abstract}
  We consider a family of polynomial analogs of the Ap\'ery numbers, which
  includes $q$-analogs of Krattenthaler--Rivoal--Zudilin and Zheng, and show
  that the supercongruences that Gessel and Mimura established for the
  Ap\'ery numbers generalize to these polynomials. Our proof relies on
  polynomial analogs of classical binomial congruences of Wolstenholme and
  Ljunggren. We further indicate that this approach generalizes to other
  supercongruence results.
\end{abstract}

\section{Introduction}

Among the many interesting properties of the {\emph{Ap\'ery numbers}}
\begin{equation}
  A (n) = \sum_{k = 0}^n \binom{n}{k}^2 \binom{n + k}{k}^2, \label{eq:apery}
\end{equation}
which are at the heart of R.~Ap\'ery's proof \cite{apery}, \cite{alf} of
the irrationality of $\zeta (3)$, are congruences with surprisingly large
modulus. Following F.~Beukers \cite{beukers-apery85} these are often
referred to as {\emph{supercongruences}}. For instance, for all primes $p
\geq 5$,
\begin{equation}
  A (p n) \equiv A (n) \quad (\operatorname{mod} p^3), \label{eq:apery-sc}
\end{equation}
as conjectured by S.~Chowla, J.~Cowles and M.~Cowles \cite{ccc-apery} and
proved by I.~Gessel \cite{gessel-super} and Y.~Mimura \cite{mimura-apery}.
When $n$ is divisible by $p$ then this congruence can be further strengthened
\cite{beukers-apery85}, \cite{coster-sc}. Indeed, the congruence $A (p^r
n) \equiv A (p^{r - 1} n)$ holds modulo $p^{3 r}$.

In this paper, we are concerned with polynomial analogs of the Ap\'ery
numbers \eqref{eq:apery} and it is our goal to demonstrate that these
polynomials share some of the remarkable arithmetic properties. In
Section~\ref{sec:A:q}, we review $q$-analogs of the Ap\'ery numbers
featuring in work of C.~Krattenthaler, T.~Rivoal and W.~Zudilin \cite{krz-q}
and D.~Zheng \cite{zheng-qapery}, as well as other natural constructions. In
particular, in Lemma~\ref{lem:Aq:krz}, we show that the $q$-numbers of
Krattenthaler, Rivoal and Zudilin, which are defined via a $q$-partial
fraction decomposition, essentially have the explicit $q$-binomial
representation
\begin{equation*}
  A_q (n) = \sum_{k = 0}^n q^{(n - k)^2} \binom{n}{k}_q^2 \binom{n +
   k}{k}_q^2,
\end{equation*}
closely resembling \eqref{eq:apery}.

In Section~\ref{sec:super}, we prove that the polynomials $A_q (n)$ satisfy
congruences modulo cubes of the $m$th cyclotomic polynomials $\Phi_m (q)$.
When specialized to $q = 1$ and $m = p$, these congruences imply the known
supercongruences \eqref{eq:apery-sc}.

\begin{corollary}
  \label{cor:A:qcong}For any integer $m$,
  \begin{equation}
    A_q (m n) \equiv A_{q^{m^2}} (n) - \frac{m^2 - 1}{12} (q^m - 1)^2 n^2 A_1
    (n) \quad (\operatorname{mod} \Phi_m (q)^3) . \label{eq:A:qcong:intro}
  \end{equation}
\end{corollary}

This result is a consequence of Theorem~\ref{thm:A:qcong}, our main theorem,
which offers a more general multivariate version. To the best of our
knowledge, congruence \eqref{eq:A:qcong:intro} is the first polynomial analog
of a supercongruence of the type \eqref{eq:apery-sc}. The congruences
\eqref{eq:apery-sc} are conjectured to hold for all Ap\'ery-like sequences
(see \cite{oss-sporadic} as well as Section~\ref{sec:concl}) but remain open
in some cases. It is our hope that understanding these congruences in the most
general setting might help shed light on the more mysterious cases. Another
reason to be interested in polynomial analogs is the availability of
additional techniques not available outside the $q$-world, illustrated, for
instance, in the very recent work \cite{gz-rama-q} of V.~Guo and W.~Zudilin.
In particular, for supercongruences of a different type, Guo and Zudilin
succeed in introducing an additional variable $a$ (so that the limit $a
\rightarrow 1$ recovers the original congruence) in such a way that the
generalized congruence holds modulo $(1 - q^n) (a - q^n) (1 - a q^n) / (1 -
q)$. The crucial benefit is that this generalized congruence can be
established modulo $(1 - q^n) / (1 - q)$, $a - q^n$ and $1 - a q^n$
individually, since these polynomials are coprime. These three individual
congruences then correspond to (exactly) evaluating the terms in question when
$q$ is an $n$th root of unity (for $(1 - q^n) / (1 - q)$) and when $a = q^{\pm
n}$ (for $a - q^n$ and $1 - a q^n$). It would be of considerable interest to
determine whether the congruences considered herein could be similarly
approached (and, to some extent, better explained) by (creatively!)
introducing an appropriate additional variable $a$ (this is referred to as
{\emph{creative microscoping}} in \cite{gz-rama-q}).

Supercongruence results for other sequences, that are based on a suitable
binomial sum representation and analogous arguments, can be generalized
likewise to the polynomial setting. We illustrate that point at the example of
the family of generalized Ap\'ery sequences
\begin{equation}
  A^{(\lambda, \mu)} (n) = \sum_{k = 0}^n \binom{n}{k}^{\lambda} \binom{n +
  k}{k}^{\mu} . \label{eq:A:x}
\end{equation}
E.~Deutsch and B.~Sagan \cite{ds-cong} showed that, if $\lambda \geq 2$
and $\mu \geq 1$, these sequences satisfy the supercongruences
\begin{equation}
  A^{(\lambda, \mu)} (p n) \equiv A^{(\lambda, \mu)} (n) \quad (\operatorname{mod}
  p^3) \label{eq:apery-sc:x}
\end{equation}
for all primes $p \geq 5$ (in fact, a generalized version of these
congruences already appears in M.~Coster's thesis \cite{coster-sc}). The
congruences \eqref{eq:apery-sc} are the special case $(\lambda, \mu) = (2,
2)$.

We prove the following polynomial analog of the congruences
\eqref{eq:apery-sc:x}. Observe how the case $(\lambda, \mu) = (2, 2)$ reduces
to the Ap\'ery number case of Corollary~\ref{cor:A:qcong}, in which case the
congruences take a particularly clean form.

\begin{theorem}
  \label{thm:A:qcong:x}For fixed integers $\lambda, \mu$ such that $\lambda
  \geq 2$ and $\mu \geq 0$, as well as a polynomial $\alpha \in
  \mathbb{Z} [n, k]$ such that $\alpha (m n, m k) = m^2 \alpha (n, k)$ and
  $\alpha (0, k) = k^2$, define
  \begin{equation}
    A^{(\lambda, \mu)}_q (n) = \sum_{k = 0}^n q^{\alpha (n, k)}
    \binom{n}{k}_q^{\lambda} \binom{n + k}{k}_q^{\mu} . \label{eq:Aq:x}
  \end{equation}
  Then, for any positive integer $m$,
  \begin{equation}
    A^{(\lambda, \mu)}_q (m n) \equiv A^{(\lambda, \mu)}_{q^{m^2}} (n) -
    \frac{m^2 - 1}{12} (q^m - 1)^2 R^{(\lambda, \mu)} (n) \quad (\operatorname{mod}
    \Phi_m (q)^3), \label{eq:A:qcong:x}
  \end{equation}
  where the numbers
  \begin{equation*}
    R^{(\lambda, \mu)} (n) = \sum_{k \geq 0} c_{n, k}^{(\lambda, \mu)}
     \binom{n}{k}^{\lambda} \binom{n + k}{k}^{\mu}, \quad c_{n, k}^{(\lambda,
     \mu)} = \left\{ \begin{array}{ll}
       n^2 + (\mu - 2) \frac{n k}{2}, & \text{if $\lambda = 2$,}\\
       ((\lambda + \mu) n - \lambda k) \frac{k}{2}, & \text{if $\lambda > 2$,}
     \end{array} \right.
  \end{equation*}
  are independent of $\alpha$ and $m$.
\end{theorem}

Our proof of this result proceeds analogously to the proof of
Theorem~\ref{thm:A:qcong} and is therefore omitted.

\begin{example}
  Note that, specializing Theorem~\ref{thm:A:qcong:x} to $(\lambda, \mu) = (2,
  0)$ and $\alpha (n, k) = k^2$, we obtain the sequence
  \begin{equation*}
    \sum_{k = 0}^n q^{k^2} \binom{n}{k}_q^2 = \binom{2 n}{n}_q
  \end{equation*}
  of central $q$-binomial coefficients. Because of the simple identity
  \begin{equation*}
    \sum_{k = 0}^n (n - k) \binom{n}{k}^2 = \frac{n}{2} \binom{2 n}{n},
  \end{equation*}
  the resulting congruences \eqref{eq:A:qcong:x} simplify to
  \begin{equation*}
    \binom{2 m n}{m n}_q \equiv \binom{2 n}{n}_{q^{m^2}} - \frac{m^2 - 1}{12}
     (q^m - 1)^2  \frac{n^2}{2} \binom{2 n}{n} \quad (\operatorname{mod} \Phi_m (q)^3)
     .
  \end{equation*}
  This matches the special case $a = 2 m$, $b = m$ of
  \eqref{eq:qljunggren:x:intro}. In the case $n = 1$ this further reduces to
  \eqref{eq:qwol:x}, which is proved in the appendix and is a crucial
  ingredient for the other polynomial congruences proved herein.
\end{example}

\section{Notations and conventions}

Throughout, all congruences, such as \eqref{eq:A:qcong:intro}, modulo a
polynomial $\varphi (q)$ are understood in the ring $\mathbb{Q} [q^{\pm 1}]$
of Laurent polynomials with rational coefficients. In other words,
\begin{equation}
  f (q) \equiv 0 \quad (\operatorname{mod} \varphi (q)) \label{eq:cong:poly}
\end{equation}
means that $f (q) = g (q) \varphi (q)$ for some Laurent polynomial $g (q) \in
\mathbb{Q} [q^{\pm 1}]$. In all congruences considered herein, $\varphi (q)$
is a monic polynomial with integer coefficients. On the other hand, we allow
$f (q)$ to have rational coefficients. For instance, the coefficients of $f
(q)$ typically involve quantities like $(m^2 - 1) / 12$ which are not integral
if $\gcd (m, 6) > 1$. Observe that, as a consequence of Gauss' lemma, if
$\varphi (q)$ and $f (q)$ both have integer coefficients and $\varphi (q)$ is
monic, then $f (q) = g (q) \varphi (q)$ implies that $g (q)$ has integer
coefficients as well. Therefore, in that case, the congruence
\eqref{eq:cong:poly} holds in the ring $\mathbb{Z} [q^{\pm 1}]$, so that, in
particular, we also have the ordinary congruence
\begin{equation}
  f (1) \equiv 0 \quad (\operatorname{mod} \varphi (1)) . \label{eq:cong:poly:1}
\end{equation}
Since our congruences are modulo powers of cyclotomic polynomials $\Phi_m
(q)$, it is useful to recall that $\Phi_m (1) = p$ if $m$ is a power of the
prime $p$. However, $\Phi_m (1) = 1$ if $m$ is not a prime power, in which
case the congruence \eqref{eq:cong:poly:1} carries no information (in contrast
to the polynomial congruence \eqref{eq:cong:poly}).

As usual, the $q$-analog of an integers $n \geq 0$ is the polynomial
$[n]_q = 1 + q + \cdots + q^{n - 1}$. Thus equipped, one introduces the
$q$-factorial as $[n]_q ! = [n]_q [n - 1]_q \cdots [1]_q$ and the $q$-binomial
coefficient as
\begin{equation}
  \binom{n}{k}_q = \frac{[n]_q !}{[k]_q ! [n - k]_q !} . \label{eq:qbin}
\end{equation}
It is easy to see from here that the $q$-binomial coefficient is a
self-reciprocal polynomial in $q$ of degree $k (n - k)$. The $q$-binomial
coefficient occurs naturally in many contexts \cite{kc-q}; for instance, if
$q$ is a prime power, then \eqref{eq:qbin} counts the number of
$k$-dimensional subspaces of an $n$-dimensional vector space over the finite
field $\mathbb{F}_q$. For our purposes, the following alternative
characterization will also be relevant. Suppose that $x$ and $y$ are
$q$-commuting in the sense that
\begin{equation}
  y x = q x y \label{eq:q:comm}
\end{equation}
(of course, in the case $q = 1$, the variables $x$ and $y$ indeed commute).
Then we have
\begin{equation}
  (x + y)^n = \sum_{k = 0}^n \binom{n}{k}_q x^k y^{n - k},
  \label{eq:qbinom:exp}
\end{equation}
extending the classical binomial expansion.

The following $q$-analog of Lucas' classical binomial congruence is proved by
G.~Olive \cite{olive-pow} and J.~D\'esarm\'enien \cite{desarmenien-q}.
More recently, it is shown in \cite{fs-qbinomial} that this result extends
to the case when $n a + b$ and $n r + s$ are allowed to be negative integers.
In \cite{abdj-cyclotomic}, B.~Adamczewski, J.~Bell, {\'E}.~Delaygue and
F.~Jouhet consider congruences modulo cyclotomic polynomials for
multidimensional $q$-factorial ratios, which generalize this and many other
Lucas-type congruences.

\begin{lemma}
  \label{lem:qlucas}For any integers $n, a, b, r, s \geq 0$ such that $b,
  s < n$,
  \begin{equation}
    \binom{a n + b}{r n + s}_q \equiv \binom{a}{r} \binom{b}{s}_q \quad
    (\operatorname{mod} \Phi_n (q)) . \label{eq:qlucas}
  \end{equation}
\end{lemma}

We next establish a polynomial analog of the classical congruence
\begin{equation}
  \binom{a p}{b p} \equiv \binom{a}{b} \quad (\operatorname{mod} p^3),
  \label{eq:ljunggren}
\end{equation}
valid for primes $p \geq 5$. Congruence \eqref{eq:ljunggren} is an
extension of Wolstenholme's congruence (which corresponds to $a = 2$ and $b =
1$) and was proved in 1952 by Ljunggren, see \cite{granville-bin97}. A
$q$-analog of \eqref{eq:ljunggren} modulo $p^2$ was proved by G.~Andrews
\cite{andrews-qcong99}, who asked for an extension modulo $p^3$. The next
congruence is such an extension.

\begin{theorem}
  \label{thm:qlunggren:x}For positive integers $n$ and nonnegative integers
  $a, b$,
  \begin{equation}
    \binom{a n}{b n}_q \equiv \binom{a}{b}_{q^{n^2}} - (a - b) b \binom{a}{b}
    \frac{n^2 - 1}{24} (q^n - 1)^2 \quad (\operatorname{mod} \Phi_n (q)^3) .
    \label{eq:qljunggren:x:intro}
  \end{equation}
\end{theorem}

This result is proved in \cite{straub-qljunggren} in the case when $n$ is
prime. Note that, by the comments after \eqref{eq:cong:poly}, setting $q = 1$
in congruence \eqref{eq:qljunggren:x:intro} indeed recovers
\eqref{eq:ljunggren} for primes $p \geq 5$. Our proof of
Theorem~\ref{thm:qlunggren:x} is a straightforward extension of the
corresponding proof given in \cite{straub-qljunggren}, where congruence
\eqref{eq:qljunggren:x:intro} is proved in the case that $n$ is a prime. We
include the details of the proof in Appendix~\ref{sec:qljunggren} for the
benefit of the reader.

\section{$q$-analogs of the Ap\'ery numbers}\label{sec:A:q}

In their investigation of irrationality results on $q$-analogs of Riemann zeta
values, C.~Krattenthaler, T.~Rivoal and W.~Zudilin \cite{krz-q} implicitly
introduce a $q$-analog $A_q^{\operatorname{KRZ}} (n)$ of the Ap\'ery numbers as
follows. In order to obtain irrationality results on
\begin{equation*}
  \zeta_q (3) = \sum_{k = 1}^{\infty} \frac{q^k (1 + q^k)}{(1 - q^k)^3},
\end{equation*}
a $q$-analog of $\zeta (3)$, they consider linear forms in $1$ and $\zeta_q
(3)$. The coefficients of these linear forms are Laurent polynomials in $q$,
and $A_q^{\operatorname{KRZ}} (n)$ is the coefficient of $\zeta_q (3)$. Specifically,
the Laurent polynomials $A_q^{\operatorname{KRZ}} (n)$ and $B_q^{\operatorname{KRZ}} (n)$ are
characterized by
\begin{equation}
  \frac{1}{\log q} \sum_{k = 1}^{\infty} \frac{\md}{\md k} \left(\frac{(q^{k - n} ; q)_n^2}{(q^k ; q)_{n + 1}^2} q^k \right) =
  A_q^{\operatorname{KRZ}} (n) \zeta_q (3) - B_q^{\operatorname{KRZ}} (n) . \label{eq:A:krz}
\end{equation}
The (Laurent) polynomials $A_q^{\operatorname{KRZ}} (n)$ have the following explicit
formula as a $q$-binomial sum which visibly reduces to Ap\'ery's binomial
sum \eqref{eq:apery} when $q = 1$.

\begin{lemma}
  \label{lem:Aq:krz}With $A_q^{\operatorname{KRZ}} (n)$ defined as in
  \eqref{eq:A:krz},
  \begin{equation}
    q^{n (2 n + 1)} A_q^{\operatorname{KRZ}} (n) = \sum_{k = 0}^n q^{(n - k)^2}
    \binom{n}{k}_q^2 \binom{n + k}{k}_q^2 . \label{eq:Aq:krz:bin}
  \end{equation}
\end{lemma}

\begin{proof}
  It is shown in \cite{krz-q} that
  \begin{equation}
    A_q^{\operatorname{KRZ}} (n) = \sum_{k = 0}^n \frac{a_q (n, k)}{q^k},
    \label{eq:A:krz:a}
  \end{equation}
  where $a_q (n, k)$ is defined via the $q$-partial fraction decomposition
  \begin{equation*}
    \frac{(q^{- n} T ; q)_n^2}{(T ; q)_{n + 1}^2} = \sum_{j = 0}^n \left(\frac{a_q (n, k)}{(1 - q^k T)^2} + \frac{b_q (n, k)}{1 - q^k T} \right),
  \end{equation*}
  where, as usual, $(a ; q)_n = (1 - a) (1 - a q) \cdots (1 - a q^{n - 1})$
  denotes the $q$-Pochhammer symbol. Therefore, by construction,
  \begin{equation*}
    a_q (n, k) = \lim_{T \rightarrow q^{- k}} (1 - q^k T)^2 \frac{(q^{- n} T
     ; q)_n^2}{(T ; q)_{n + 1}^2} = \left[ \lim_{T \rightarrow q^{- k}} (1 -
     q^k T) \frac{(q^{- n} T ; q)_n}{(T ; q)_{n + 1}} \right]^2 .
  \end{equation*}
  Expanding the $q$-Pochhammer symbols, canceling the factor $1 - q^k T$, and
  setting $T = q^{- k}$ in the resulting expression, we find that
  \begin{equation*}
    \lim_{T \rightarrow q^{- k}} (1 - q^k T) \frac{(q^{- n} T ; q)_n}{(T ;
     q)_{n + 1}} = \frac{(q^{- 1} ; q^{- 1})_{n + k}}{(q^{- 1} ; q^{- 1})_k^2
     (q ; q)_{n - k}} .
  \end{equation*}
  Using the transformation formula
  \begin{equation*}
    (q ; q)_n = (- 1)^n q^{\binom{n + 1}{2}} (q^{- 1} ; q^{- 1})_n,
  \end{equation*}
  we observe that
  \begin{equation*}
    \frac{(q^{- 1} ; q^{- 1})_{n + k}}{(q^{- 1} ; q^{- 1})_k^2 (q ; q)_{n -
     k}} = (- 1)^{n + k} q^{2 \binom{k + 1}{2} - \binom{n + k + 1}{2}}
     \binom{n}{k}_q \binom{n + k}{k}_q .
  \end{equation*}
  We hence conclude that
  \begin{equation*}
    a_q (n, k) = q^{2 k (k + 1) - (n + k) (n + k + 1)} \binom{n}{k}_q^2
     \binom{n + k}{k}_q^2,
  \end{equation*}
  which, in combination with \eqref{eq:A:krz:a}, implies
  \eqref{eq:Aq:krz:bin}.
\end{proof}

Investigating $q$-analogs of a result of S.~Ahlgren and K.~Ono
\cite{ahlgren-ono-apery} and W.~Chu \cite{chu-apery} pertaining to a
modular supercongruence that was conjectured by F.~Beukers
\cite{beukers-apery87} (see Section~\ref{sec:concl} and \eqref{eq:A:scm} for
further details), D.~Zheng \cite{zheng-qapery} introduces the $q$-Ap\'ery
numbers
\begin{equation}
  \sum_{k = 0}^n q^{k (k - 2 n)} \binom{n}{k}_q^2 \binom{n + k}{k}_q^2 = q^{n
  (n + 1)} A_q^{\operatorname{KRZ}} (n) . \label{eq:Aq:z}
\end{equation}
and obtains the identity
\begin{equation}
  \sum_{k = 0}^n q^{k (k - 2 n)} \binom{n}{k}_q^2 \binom{n + k}{k}_q^2 \{ 2
  H_q (k) - H_q (n + k) - q H_{1 / q} (n - k) \} = 0, \label{eq:Aq:chu}
\end{equation}
involving the $q$-harmonic numbers
\begin{equation*}
  H_q (n) = \sum_{k = 1}^n \frac{1}{[k]_q} .
\end{equation*}
In the remainder of this section, we illustrate that the $q$-Ap\'ery numbers
considered by Krattenthaler--Rivoal--Zudilin and Zheng can also be obtained as
special cases of a natural construction based on introducing $q$-commuting
variables.

It was shown in \cite{s-apery} that the Ap\'ery numbers are the diagonal
Taylor coefficients of the rational function
\begin{equation}
  \frac{1}{(1 - x_1 - x_2) (1 - x_3 - x_4) - x_1 x_2 x_3 x_4} = \sum_{n_1,
  n_2, n_3, n_4 \geq 0} A (\boldsymbol{n}) \boldsymbol{x}^{\boldsymbol{n}} .
  \label{eq:apery-rf}
\end{equation}
Here, $\boldsymbol{n}= (n_1, \ldots, n_4)$ and $\boldsymbol{x}^{\boldsymbol{n}} =
x_1^{n_1} x_2^{n_2} x_3^{n_3} x_4^{n_4}$. The multivariate Ap\'ery numbers
$A (\boldsymbol{n})$ featuring in this expansion have the explicit
representation
\begin{equation}
  A (\boldsymbol{n}) = \sum_{k \in \mathbb{Z}} \binom{n_1}{k} \binom{n_3}{k}
  \binom{n_1 + n_2 - k}{n_1} \binom{n_3 + n_4 - k}{n_3} . \label{eq:apery4x}
\end{equation}
\begin{example}
  \label{eg:Aq:sym}In search for a natural $q$-analog of these multivariate
  Ap\'ery numbers, it is reasonable to consider sums of $q$-binomial
  coefficients of the form
  \begin{equation*}
    \sum_{k \geq 0} q^{h (\boldsymbol{n}; k)} \binom{n_1}{k}_q
     \binom{n_3}{k}_q \binom{n_1 + n_2 - k}{n_1}_q \binom{n_3 + n_4 -
     k}{n_3}_q,
  \end{equation*}
  where $h (\boldsymbol{n}; k)$ is a polynomial in $\boldsymbol{n}$ and $k$,
  taking integer values. We note that a particularly natural choice is $h
  (\boldsymbol{n}; k) = k^2$ because then the resulting $q$-numbers are
  self-reciprocal polynomials in $q$ of degree $n_1 n_2 + n_3 n_4$. To see
  this, observe that
  \begin{equation*}
    \deg \binom{n_1}{k}_q \binom{n_3}{k}_q \binom{n_1 + n_2 - k}{n_1}_q
     \binom{n_3 + n_4 - k}{n_3}_q = n_1 n_2 + n_3 n_4 - 2 k^2 .
  \end{equation*}
  Consequently, since each $q$-binomial is self-reciprocal, $q^{k^2}$ times
  this product of four $q$-binomials is a polynomial $z (q)$ satisfying $z (q)
  = q^{n_1 n_2 + n_3 n_4} z (1 / q)$.
\end{example}

A more organic approach to obtaining the $q$-analog in
Example~\ref{eg:Aq:sym}, as well as other variations, is based on the
observation made in \cite{s-apery} that, by MacMahon's Master Theorem
\cite{macmahon-comb}, the power series expansion \eqref{eq:apery-rf}
defining the multivariate Ap\'ery numbers $A (\boldsymbol{n})$ is equivalent
to
\begin{equation*}
  A (\boldsymbol{n}) = [\boldsymbol{x}^{\boldsymbol{n}}] (x_1 + x_2 + x_3)^{n_1}
   (x_1 + x_2)^{n_2} (x_3 + x_4)^{n_3} (x_2 + x_3 + x_4)^{n_4},
\end{equation*}
where $A (\boldsymbol{n})$ is represented as the coefficient of
$\boldsymbol{x}^{\boldsymbol{n}}$ in a certain polynomial.

In light of the $q$-binomial expansion \eqref{eq:qbinom:exp}, it is natural to
define a $q$-analog of $A (\boldsymbol{n})$ by extracting the coefficient of
$\boldsymbol{x}^{\boldsymbol{n}}$ in the product of $(x_1 + x_2 + x_3)^{n_1}$,
$(x_1 + x_2)^{n_2}$, $(x_3 + x_4)^{n_3}$ and $(x_2 + x_3 + x_4)^{n_4}$, where
now the variables $x_i$ are assumed to be $q$-commuting in the spirit of
\eqref{eq:q:comm}. Depending on the exact choices, including the order of the
factors, one obtains different $q$-analogs.

\begin{example}
  \label{eg:Aq:mm:nonsym}Suppose that, for all $i < j$, the variables $x_i$
  commute according to $x_j x_i = q x_i x_j$. Let $A_q (\boldsymbol{n})$ denote
  the coefficient of $\boldsymbol{x}^{\boldsymbol{n}}$ in the product
  \begin{equation*}
    (x_1 + x_2 + x_3)^{n_1} (x_1 + x_2)^{n_2} (x_3 + x_4)^{n_3} (x_2 + x_3 +
     x_4)^{n_4} .
  \end{equation*}
  Then, using the $q$-binomial expansion \eqref{eq:qbinom:exp}, we obtain that
  \begin{equation*}
    A_q (\boldsymbol{n}) = \sum_{k = 0}^{\min (\boldsymbol{n})} q^{k (n_2 + n_3 +
     k)} \binom{n_1}{k}_q \binom{n_3}{k}_q \binom{n_1 + n_2 - k}{n_1}_q
     \binom{n_3 + n_4 - k}{n_4}_q .
  \end{equation*}
\end{example}

\begin{example}
  \label{eg:Aq:mm:sym}On the other hand, if we modify the previous example to
  let $A_q (\boldsymbol{n})$ denote the coefficient of
  $\boldsymbol{x}^{\boldsymbol{n}}$ in the product
  \begin{equation*}
    (x_1 + x_2)^{n_2} (x_1 + x_2 + x_3)^{n_1} (x_2 + x_3 + x_4)^{n_4} (x_3 +
     x_4)^{n_3},
  \end{equation*}
  then we find that
  \begin{equation*}
    A_q (\boldsymbol{n}) = \sum_{k = 0}^{\min (\boldsymbol{n})} q^{k^2}
     \binom{n_1}{k}_q \binom{n_3}{k}_q \binom{n_1 + n_2 - k}{n_1}_q \binom{n_3
     + n_4 - k}{n_4}_q,
  \end{equation*}
  which is the symmetric $q$-analog corresponding to the choice $h
  (\boldsymbol{n}; k) = k^2$ in Example~\ref{eg:Aq:sym}. Specializing to $n_1 =
  n_2 = n_3 = n_4 = n$, the polynomial $A_q (\boldsymbol{n})$ equals
  \eqref{eq:Aq:krz:bin}, and we recover the $q$-Ap\'ery numbers considered
  by Krattenthaler--Rivoal--Zudilin and Zheng.
\end{example}

In the next section, we will show that all of the $q$-analogs, including the
multivariate ones from Examples~\ref{eg:Aq:mm:nonsym} and \ref{eg:Aq:mm:sym},
discussed in this section satisfy supercongruences generalizing the
congruences \eqref{eq:apery-sc}.

\section{Proof of the supercongruences}\label{sec:super}

This section is concerned with proving the following multivariate
generalization of Corollary~\ref{cor:A:qcong}. This generalization is inspired
by the corresponding classical congruences, which were shown in
\cite{s-apery} and which can be obtained from Theorem~\ref{thm:A:qcong} by
specializing $q = 1$.

\begin{theorem}
  \label{thm:A:qcong}For $\alpha \in \mathbb{Z} [\boldsymbol{n}, k]$ such that
  $\alpha (m\boldsymbol{n}, m k) = m^2 \alpha (\boldsymbol{n}, k)$ and $\alpha
  (\boldsymbol{0}, k) = k^2$, define
  \begin{equation}
    A_q^{(\alpha)} (\boldsymbol{n}) = \sum_{k \in \mathbb{Z}} q^{\alpha
    (\boldsymbol{n}, k)} \binom{n_1}{k}_q \binom{n_3}{k}_q \binom{n_1 + n_2 -
    k}{n_1}_q \binom{n_3 + n_4 - k}{n_3}_q . \label{eq:A:mult}
  \end{equation}
  Then, for any positive integer $m$,
  \begin{equation}
    A_q^{(\alpha)} (m\boldsymbol{n}) \equiv A_{q^{m^2}}^{(\alpha)}
    (\boldsymbol{n}) - \frac{m^2 - 1}{12} (q^m - 1)^2 R (\boldsymbol{n}) \quad
    (\operatorname{mod} \Phi_m (q)^3), \label{eq:A:qcong}
  \end{equation}
  where the numbers
  \begin{equation*}
    R (\boldsymbol{n}) = \frac{n_1 n_2 + n_3 n_4}{2} A_1^{(\alpha)}
     (\boldsymbol{n})
  \end{equation*}
  are independent of $\alpha$ and $m$.
\end{theorem}

\begin{proof}
  We have
  \begin{equation*}
    A_q^{(\alpha)} (\boldsymbol{n}) = \sum_{k \geq 0} B_q (\boldsymbol{n};
     k)
  \end{equation*}
  with
  \begin{equation*}
    B_q (\boldsymbol{n}; k) = q^{\alpha (\boldsymbol{n}, k)} \binom{n_1}{k}_q
     \binom{n_3}{k}_q \binom{n_1 + n_2 - k}{n_1}_q \binom{n_3 + n_4 -
     k}{n_3}_q .
  \end{equation*}
  Following the approach of \cite{gessel-super}, we write
  \begin{equation}
    A_q^{(\alpha)} (m\boldsymbol{n}) = S_q^{(1)} (\boldsymbol{n}) + S_q^{(2)}
    (\boldsymbol{n}), \label{eq:A:S12}
  \end{equation}
  where
  \begin{equation*}
    S_q^{(1)} (\boldsymbol{n}) = \sum_{m|k} B_q (m\boldsymbol{n}; k), \quad
     S_q^{(2)} (\boldsymbol{n}) = \sum_{j = 1}^{m - 1} \sum_{k \geq 0} B_q
     (m\boldsymbol{n}; m k + j) .
  \end{equation*}
  Applying Theorem~\ref{thm:qlunggren:x}, the $q$-analog of Ljunggren's
  binomial congruence, to each of the binomial coefficients in the summand of
  $S_q^{(1)} (\boldsymbol{n})$ and simplifying (keeping in mind that $q^m \equiv
  1$ modulo $\Phi_m (q)$), we obtain that
  \begin{equation}
    S_q^{(1)} (\boldsymbol{n}) \equiv A_{q^{m^2}}^{(\alpha)} (\boldsymbol{n}) -
    \frac{m^2 - 1}{12} (q^m - 1)^2 R^{(1)} (\boldsymbol{n}) \quad (\operatorname{mod}
    \Phi_m (q)^3), \label{eq:S1}
  \end{equation}
  where
  \begin{eqnarray*}
    R^{(1)} (\boldsymbol{n}) & = & \sum_{k \geq 0} \left[ \frac{(n_1 - k)
    k}{2} + \frac{(n_3 - k) k}{2} + \frac{(n_2 - k) n_1}{2} + \frac{(n_4 - k)
    n_3}{2} \right] C (\boldsymbol{n}; k)\\
    & = & \sum_{k \geq 0} \left[ \frac{n_1 n_2 + n_3 n_4}{2} - k^2
    \right] C (\boldsymbol{n}; k)
  \end{eqnarray*}
  with
  \begin{equation}
    C (\boldsymbol{n}; k) = \binom{n_1}{k} \binom{n_3}{k} \binom{n_1 + n_2 -
    k}{n_1} \binom{n_3 + n_4 - k}{n_3} . \label{eq:A:summand}
  \end{equation}
  Note that $q^m \equiv 1$ modulo $\Phi_m (q)$, which is why the term
  $q^{\alpha (m\boldsymbol{n}, m k)} = q^{m^2 \alpha (\boldsymbol{n}, k)}$ in $B_q
  (m\boldsymbol{n}; m k)$ does not contribute to $R^{(1)} (\boldsymbol{n})$.
  
  It therefore remains to consider $S_q^{(2)} (\boldsymbol{n})$ modulo $\Phi_m
  (q)^3$. It is proved in Proposition~\ref{prop:qbin:cong:2} that
  \begin{equation*}
    \binom{m n}{m k + j}_q \equiv (- 1)^{j - 1} q^{(j - 1) (m - j / 2)}
     \frac{[m n]_q}{[j]_q} \binom{n - 1}{k} \quad (\operatorname{mod} \Phi_m (q)^2) .
  \end{equation*}
  On the other hand, it follows from Lemma~\ref{lem:qlucas}, the $q$-analog of
  Lucas' congruence, that, for $0 < j \leq m$,
  \begin{equation*}
    \binom{m (n + k) - j}{m n}_q = \binom{m (n + k - 1) + (m - j)}{m n}_q
     \equiv \binom{n + k - 1}{n} \quad (\operatorname{mod} \Phi_m (q)) .
  \end{equation*}
  Note that the congruence holds trivially in the case $n = k = 0$.
  
  Thus, modulo $\Phi_m (q)^3$,
  \begin{eqnarray}
    S_q^{(2)} (\boldsymbol{n}) & \equiv & \sum_{j = 1}^{m - 1} \sum_{k \geq
    0} q^{\alpha (m\boldsymbol{n}, m k + j)} q^{(j - 1) (2 m - j)} \frac{[m
    n_1]_q [m n_3]_q}{[j]_q^2} \nonumber\\
    &  & \times \binom{n_1 - 1}{k} \binom{n_3 - 1}{k} \binom{n_1 + n_2 - k -
    1}{n_1} \binom{n_3 + n_4 - k - 1}{n_3} \nonumber\\
    & \equiv & \frac{[m n_1]_q [m n_3]_q}{n_1 n_3} \sum_{j = 1}^{m - 1}
    \frac{q^j}{[j]_q^2} \sum_{k \geq 0} (k + 1)^2 C (\boldsymbol{n}; k +
    1),  \label{eq:S2:0}
  \end{eqnarray}
  with $C (\boldsymbol{n}; k)$ as in \eqref{eq:A:summand}. For the second
  congruence, we used the presence of the term $[m n_1]_q [m n_3]_q$, which is
  divisible by $\Phi_m (q)^2$, together with
  \begin{equation*}
    q^{\alpha (m\boldsymbol{n}, m k + j)} q^{(j - 1) (2 m - j)} \equiv
     q^{\alpha (\boldsymbol{0}, j)} q^{j - j^2} \equiv q^j \quad (\operatorname{mod}
     \Phi_m (q)),
  \end{equation*}
  to reduce the powers of $q$ modulo $\Phi_m (q)^3$.
  
  Note that it follows from the simple observation
  \begin{equation*}
    \sum_{j = 1}^{m - 1} \frac{q^j}{[j]_q^2} = \sum_{j = 1}^{m - 1}
     \frac{1}{[j]_q^2} + (q - 1) \sum_{j = 1}^{m - 1} \frac{1}{[j]_q},
  \end{equation*}
  combined with the congruences \eqref{eq:sp1} and \eqref{eq:sp2}, that, for
  all integers $m$,
  \begin{equation*}
    \sum_{j = 1}^{m - 1} \frac{q^j}{[j]_q^2} \equiv - \frac{m^2 - 1}{12} (q -
     1)^2 \quad (\operatorname{mod} \Phi_m (q)) .
  \end{equation*}
  Further using that
  \begin{equation*}
    (q - 1)^2 [m n_1]_q [m n_3]_q \equiv (q^m - 1)^2 n_1 n_3 \quad
     (\operatorname{mod} \Phi_m (q)^3),
  \end{equation*}
  we conclude from \eqref{eq:S2:0} that
  \begin{equation}
    S_q^{(2)} (\boldsymbol{n}) \equiv - \frac{m^2 - 1}{12} (q^m - 1)^2 \sum_{k
    \geq 0} k^2 C (\boldsymbol{n}; k) \quad (\operatorname{mod} \Phi_m (q)^3) .
    \label{eq:S2}
  \end{equation}
  Combining \eqref{eq:A:S12}, \eqref{eq:S1} and \eqref{eq:S2}, we therefore
  have
  \begin{equation*}
    A_q^{(\alpha)} (m\boldsymbol{n}) \equiv A_{q^{m^2}}^{(\alpha)}
     (\boldsymbol{n}) - \frac{m^2 - 1}{12} (q^m - 1)^2 \frac{n_1 n_2 + n_3
     n_4}{2} \sum_{k \geq 0} C (\boldsymbol{n}; k) \quad (\operatorname{mod} \Phi_m
     (q)^3),
  \end{equation*}
  which is the claimed congruence \eqref{eq:A:qcong}.
\end{proof}

\begin{proposition}
  \label{prop:qbin:cong:2}For nonnegative integers $j, k, m, n$, with $0 < j <
  m$,
  \begin{equation*}
    \binom{m n}{m k + j}_q \equiv (- 1)^{j - 1} q^{(j - 1) (m - j / 2)}
     \frac{[m n]_q}{[j]_q} \binom{n - 1}{k} \quad (\operatorname{mod} \Phi_m (q)^2) .
  \end{equation*}
\end{proposition}

\begin{proof}
  We have
  \begin{eqnarray*}
    \binom{m n}{m k + j}_q & = & \frac{[m n]_q}{[m k + j]_q} \binom{m n - 1}{m
    k + j - 1}_q\\
    & \equiv & \frac{[m n]_q}{[j]_q} \binom{m n - 1}{m k + j - 1}_q \quad
    (\operatorname{mod} \Phi_m (q)^2),
  \end{eqnarray*}
  because $[a b]_q = [a]_{q^b} [b]_q$ and $[a + b]_q = [a]_q + q^a [b]_q$.
  Moreover, using Lemma~\ref{lem:qlucas}, the $q$-analog of Lucas' binomial
  congruence, if $0 < j < m$,
  \begin{eqnarray*}
    \binom{m n - 1}{m k + j - 1}_q & = & \binom{m (n - 1) + m - 1}{m k + j -
    1}_q\\
    & \equiv & \binom{n - 1}{k} \binom{m - 1}{j - 1}_q \quad (\operatorname{mod}
    \Phi_m (q)) .
  \end{eqnarray*}
  On the other hand,
  \begin{eqnarray*}
    \binom{m - 1}{j - 1}_q & = & \prod_{k = 1}^{j - 1} \frac{[m - k]_q}{[k]_q}
    = \prod_{k = 1}^{j - 1} \frac{[m]_q - q^{m - k} [k]_q}{[k]_q}\\
    & \equiv & (- 1)^{j - 1} q^{(j - 1) (2 m - j) / 2} \quad (\operatorname{mod}
    \Phi_m (q)) .
  \end{eqnarray*}
  Combined, the claim follows.
\end{proof}

\section{Outlook and open problems}\label{sec:concl}

A major motivation for this paper is the observation of R.~Osburn and B.~Sahu
\cite{os-mf} that all Ap\'ery-like numbers appear to satisfy
supercongruences. However, despite recent progress \cite{oss-sporadic}, it
remains open to show that, for instance, the {\emph{Almkvist--Zudilin
numbers}} \cite[sequence (4.12){\hspace{0.25em}}($\delta$)]{asz-clausen},
\cite{cz-apery}, \cite{ccs-apery}
\begin{equation}
  Z (n) = \sum_{k = 0}^{\lfloor n / 3 \rfloor} (- 3)^{n - 3 k} \frac{(n + k)
  !}{(n - 3 k) !k!^4}, \label{eq:az}
\end{equation}
satisfy the supercongruence
\begin{equation}
  Z (p^r m) \equiv Z (p^{r - 1} m) \quad (\operatorname{mod} p^{3 r}) \label{eq:az-sc}
\end{equation}
for all primes $p \geq 3$. While the case $r > 1$ remains open,
T.~Amdeberhan and R.~Tauraso \cite{at-az} recently proved the case $r = 1$
of these congruences. It is not obvious how to introduce a $q$-analog of the
Almkvist--Zudilin numbers $Z (n)$ in such a way that the congruence $Z (p m)
\equiv Z (m)$ modulo $p^3$ has a polynomial analog of the kind proved in
Corollary~\ref{cor:A:qcong} for the $q$-Ap\'ery numbers. On the other hand,
finding a suitable $q$-analog of the numbers $Z (n)$ might provide some
insight towards approaching the conjectured congruences \eqref{eq:az-sc} (and,
even more optimistically, might lead to a better understanding of why such
supercongruences hold for all of these sequences). In addition, as explained
after Corollary~\ref{cor:A:qcong}, additional techniques exist for approaching
polynomial supercongruences, including the {\emph{creative microscoping}}
approach recently discovered by V.~Guo and W.~Zudilin \cite{gz-rama-q}, who
apply it with remarkable success to supercongruences in the context of
Ramanujan-type formulae for $1 / \pi$. It would be interesting to determine
whether this approach, which involves (creatively!) introducing an additional
parameter, can be applied to the congruences studied herein.

Returning to the specific case of the Almkvist--Zudilin numbers $Z (n)$, it
was observed in \cite{s-apery} that they are the diagonal coefficients of
the particularly simple rational function
\begin{equation}
  \frac{1}{1 - (x_1 + x_2 + x_3 + x_4) + 27 x_1 x_2 x_3 x_4} \label{eq:az-rf}
  .
\end{equation}
Recall that, similarly, the Ap\'ery numbers are the diagonal coefficients of
the rational function \eqref{eq:apery-rf}, and that we were able to use that
fact to introduce natural $q$-analogs. However, in the case of the
Almkvist--Zudilin numbers, our attempts to define a $q$-analog satisfying
supercongruences like \eqref{eq:A:qcong:intro} have not been successful. A
crucial difference is that MacMahon's Master Theorem does not apply as it did
in the case of the Ap\'ery number rational function.

It was shown by F.~Beukers \cite{beukers-apery87} that the Ap\'ery numbers
$A (n)$ occur as the coefficients when expanding the modular form
\begin{equation}
  \frac{\eta^7 (2 \tau) \eta^7 (3 \tau)}{\eta^5 (\tau) \eta^5 (6 \tau)} =
  \sum_{n \geq 0} A (n)  \left(\frac{\eta (\tau) \eta (6 \tau)}{\eta (2
  \tau) \eta (3 \tau)} \right)^{12 n} . \label{eq:apery-mp}
\end{equation}
Note that the expansion \eqref{eq:apery-mp} is in terms of a modular function.
As usual, $\eta (\tau)$ denotes the Dedekind eta function $\eta (\tau) =
e^{\pi i \tau / 12} \prod_{n \geq 1} (1 - e^{2 \pi i n \tau})$. Can this
connection to modular forms be suitably extended to a $q$-analog of the
Ap\'ery numbers?

In another direction, Beukers \cite{beukers-apery87} related the Ap\'ery
numbers to the $p$th Fourier coefficient $a (p)$ of the modular form $\eta^4
(2 \tau) \eta^4 (4 \tau)$, where $\eta (\tau)$ is the usual Dedekind eta
function. He conjectured the congruence
\begin{equation}
  A \left(\frac{p - 1}{2} \right) \equiv a (p) \quad (\operatorname{mod} p^2),
  \label{eq:A:scm}
\end{equation}
and proved that \eqref{eq:A:scm} holds modulo $p$. The \textit{modular
supercongruence} \eqref{eq:A:scm} was later proved by S.~Ahlgren and K.~Ono
\cite{ahlgren-ono-apery}, who reduce it to the identity
\begin{equation}
  \sum_{k = 1}^n \binom{n}{k}^2 \binom{n + k}{k}^2 \{ 1 + 2 k H_{n + k} + 2 k
  H_{n - k} - 4 k H_k \} = 0, \label{eq:A:scm:h}
\end{equation}
involving the harmonic sums $H_k = 1 / 1 + 1 / 2 + \cdots + 1 / k$, which may
be confirmed by the WZ method \cite{aeoz-beukers}. A classical proof of
\eqref{eq:A:scm:h}, by means of a partial fraction decomposition, has been
given by W.~Chu \cite{chu-apery}, who proves that
\begin{equation*}
  \frac{x (1 - x)_n^2}{(x)_{n + 1}^2} = \frac{1}{x} + \sum_{k = 1}^n
   \binom{n}{k}^2 \binom{n + k}{k}^2 \left\{ \frac{- k}{(x + k)^2} + \frac{1 +
   2 k H_{n + k} + 2 k H_{n - k} - 4 k H_k}{x + k} \right\},
\end{equation*}
which reduces to \eqref{eq:A:scm:h} in the limit. A $q$-analog of Chu's
identity is obtained by D.~Zheng \cite{zheng-qapery}. In the limit, Zheng's
result specializes to \eqref{eq:Aq:chu}, which is a $q$-variation of
\eqref{eq:A:scm:h} (taking a different limit, Zheng \
\cite[Corollary~3]{zheng-qapery} also gives a literal $q$-analog of
\eqref{eq:A:scm:h}). It would be of considerable interest to find an
appropriate setting for $q$-extending the supercongruence \eqref{eq:A:scm}.
Some pointers might be found in the recent work \cite{gz-rama-q} of V.~Guo
and W.~Zudilin on $q$-analogs of Ramanujan-type formulae for $1 / \pi$.

\appendix\section{Proof of
Theorem~\ref{thm:qlunggren:x}}\label{sec:qljunggren}

The proof of Theorem~\ref{thm:qlunggren:x} presented in this appendix is a
straightforward extension of the corresponding proof given by the author in
\cite{straub-qljunggren}, from where most parts are taken literally.

It is well-known that the $q$-binomial coefficients have the following
square-free factorization into cyclotomic polynomials. A less explicit
statement appears in \cite[Lemma~2]{clark-qbin95}, where the conclusion is
emphasized that $\binom{n}{k}_q$ is a square-free polynomial which, for $k =
1, 2, \ldots, n - 1$, is divisible by $\Phi_n (q)$.

\begin{lemma}
  For nonnegative integers $n, k$,
  \begin{equation}
    \binom{n}{k}_q = \prod_d \Phi_d (q), \label{eq:qbin:fac}
  \end{equation}
  where the product is taken over those $d \in \{2, 3, \ldots, n\}$ for which
  $\lfloor n / d \rfloor - \lfloor k / d \rfloor - \lfloor (n - k) / d \rfloor
  = 1$.
\end{lemma}

\begin{proof}
  By definition, the $q$-number $[n]_q$ has the factorization
  \begin{equation*}
    [n]_q = \prod_{d|n, d > 1} \Phi_d (q) .
  \end{equation*}
  Therefore,
  \begin{equation*}
    [n]_q ! = \prod_{m = 1}^n \prod_{d|m, d > 1} \Phi_d (q) = \prod_{d = 2}^n
     \Phi_d (q)^{\lfloor n / d \rfloor},
  \end{equation*}
  as well as
  \begin{equation*}
    \binom{n}{k}_q = \frac{[n]_q !}{[k]_q ! [n - k]_q !} = \prod_{d = 2}^n
     \Phi_d (q)^{\lfloor n / d \rfloor - \lfloor k / d \rfloor - \lfloor (n -
     k) / d \rfloor} .
  \end{equation*}
  Clearly, $\lfloor n / d \rfloor - \lfloor k / d \rfloor - \lfloor (n - k) /
  d \rfloor \in \{0, 1\}$, so that we obtain \eqref{eq:qbin:fac}.
\end{proof}

In preparation for the proof of Theorem~\ref{thm:qlunggren:x}, we show the
following special case, to which the general case is then reduced.

\begin{lemma}
  \label{lem:qwol:x}For integers $n \geq 0$,
  \begin{equation}
    \binom{2 n}{n}_q \equiv [2]_{q^{n^2}} - \frac{n^2 - 1}{12} (q^n - 1)^2
    \quad (\operatorname{mod} \Phi_n (q)^3) . \label{eq:qwol:x}
  \end{equation}
\end{lemma}

\begin{proof}
  Using that $[a n]_q = [a]_{q^n} [n]_q$ and $[n + k]_q = [n]_q + q^n  [k]_q$,
  we compute
  \begin{equation*}
    \binom{2 n}{n}_q = \frac{[2 n]_q}{[n]_q} \prod_{k = 1}^{n - 1} \frac{[n +
     k]_q}{[k]_q} = [2]_{q^n} \prod_{k = 1}^{n - 1} \left(\frac{[n]_q}{[k]_q}
     + q^n \right),
  \end{equation*}

  which modulo $\Phi_n (q)^3$ reduces to (note that the terms $[k]_q$, with $k
  < n$, are relatively prime to $\Phi_n (q)$)
  \begin{equation}
    [2]_{q^n} \left(q^{(n - 1) n} + q^{(n - 2) n} \sum_{0 < i < n}
    \frac{[n]_q}{[i]_q} + q^{(n - 3) n} \sum_{0 < i < j < n} \frac{[n]_q
    [n]_q}{[i]_q [j]_q} \right) . \label{eq-2ppharmonic}
  \end{equation}
  Generalizing an earlier result of Andrews \cite{andrews-qcong99}, Shi and
  Pan prove in \cite{shipan-qwolst07} and \cite{pan-wolstenholme} that,
  for all positive integers $n$,
  \begin{equation}
    \sum_{0 < i < n} \frac{1}{[i]_q} \equiv - \frac{n - 1}{2} (q - 1) +
    \frac{n^2 - 1}{24} (q - 1)^2 [n]_q \quad (\operatorname{mod} \Phi_n (q)^2) .
    \label{eq:sp1}
  \end{equation}
  Note that this is a $q$-analog of Wolstenholme's well-known harmonic series
  congruence. Further, we have
  \begin{equation}
    \sum_{0 < i < n} \frac{1}{[i]_q^2} \equiv - \frac{(n - 1) (n - 5)}{12} (q
    - 1)^2 \quad (\operatorname{mod} \Phi_n (q)), \label{eq:sp2}
  \end{equation}
  which is shown, for prime $n$, in \cite{shipan-qwolst07} and the proof
  extends readily to the case of composite $n$. Combining the two congruences
  \eqref{eq:sp1} and \eqref{eq:sp2}, we conclude that
  \begin{equation}
    \sum_{0 < i < j < n} \frac{1}{[i]_q [j]_q} \equiv \frac{(n - 1) (n -
    2)}{6} (q - 1)^2 \quad (\operatorname{mod} \Phi_n (q)) . \label{eq:sp3}
  \end{equation}
  Applying the congruences \eqref{eq:sp1} and \eqref{eq:sp3} to
  \eqref{eq-2ppharmonic}, we obtain
  \begin{eqnarray*}
    \binom{2 n}{n}_q & \equiv & (1 + q^n) \left(q^{(n - 1) n} + q^{(n - 2) n}
    \left(- \frac{n - 1}{2} (q^n - 1) + \frac{n^2 - 1}{24} (q^n - 1)^2
    \right) \right.\\
    &  & \left. + q^{(n - 3) n} \frac{(n - 1) (n - 2)}{6} (q^n - 1)^2 \right)
    \quad (\operatorname{mod} \Phi_n (q)^3) .
  \end{eqnarray*}
  In order to reduce the terms $q^{m n}$ modulo the appropriate power of
  $\Phi_n (q)$, we use the binomial expansion
  \begin{equation*}
    q^{m n} = ((q^n - 1) + 1)^m = \sum_{k = 0}^m \binom{m}{k} (q^n - 1)^k
     ,
  \end{equation*}
  This results in the simplified congruence
  \begin{equation}
    \binom{2 n}{n}_q \equiv 2 + n (q^n - 1) + \frac{(n - 1) (5 n - 1)}{12}
    (q^n - 1)^2 \quad (\operatorname{mod} \Phi_n (q)^3) . \label{eq:qwol:x2}
  \end{equation}

  It only remains to note that
  \begin{equation*}
    [2]_{q^{n^2}} = 1 + q^{n^2} \equiv 2 + n (q^n - 1) + \frac{n (n - 1)}{2}
     (q^n - 1)^2 \quad (\operatorname{mod} \Phi_n (q)^3),
  \end{equation*}
  in order to see that \eqref{eq:qwol:x2} is equivalent to the claimed
  congruence \eqref{eq:qwol:x}.
\end{proof}

We are now in a comfortable position to prove Theorem~\ref{thm:qlunggren:x},
which is restated below for the convenience of the reader.

\begin{theorem}
  For positive integers $n$ and nonnegative integers $a, b$,
  \begin{equation}
    \binom{a n}{b n}_q \equiv \binom{a}{b}_{q^{n^2}} - (a - b) b \binom{a}{b}
    \frac{n^2 - 1}{24} (q^n - 1)^2 \quad (\operatorname{mod} \Phi_n (q)^3) .
    \label{eq:qljunggren:x}
  \end{equation}
\end{theorem}

\begin{proof}
  Observe that the two sides of \eqref{eq:qljunggren:x} are trivially equal
  when $a = 0$ or $a = 1$. In the following, we therefore assume that $a
  \geq 2$. Then, as shown in \cite{clark-qbin95}, it follows from a
  $q$-analog of the Chu-Vandermonde convolution formula that
  \begin{equation}
    \binom{a n}{b n}_q = \sum_{c_1 + \ldots + c_a = b n} q^{^{n \sum_{1
    \leq i \leq a} (i - 1) c_i - \sum_{1 \leq i < j \leq
    a} c_i c_j}} \binom{n}{c_1}_q \binom{n}{c_2}_q \cdots \binom{n}{c_a}_q .
    \label{eq:qchu}
  \end{equation}
  By \eqref{eq:qbin:fac}, the binomial coefficient $\binom{n}{k}_q$ is
  divisible by $\Phi_n (q)$ except in the boundary cases $k = 0$ and $k = n$.
  Therefore, the summands on the right-hand side of \eqref{eq:qchu} which are
  not divisible by $\Phi_n (q)^2$ correspond to the $c_i$ taking only the
  values $0$ and $n$. Since each such summand is determined by the indices $1
  \leq j_1 < j_2 < \ldots < j_b \leq a$ for which $c_i = n$, the
  total contribution of these terms is
  \begin{equation*}
    \sum_{1 \leq j_1 < \ldots < j_b \leq a} q^{n^2 \sum_{k = 1}^b
     (j_k - 1)  - n^2 \binom{b}{2}} = \sum_{0 \leq i_1 \leq
     \ldots \leq i_b \leq a - b} q^{n^2 \sum_{k = 1}^b i_k} =
     \binom{a}{b}_{q^{n^2}} .
  \end{equation*}
  Note that we have arrived at \eqref{eq:qljunggren:x} modulo $\Phi_n (q)^2$,
  that is
  \begin{equation*}
    \binom{a n}{b n}_q \equiv \binom{a}{b}_{q^{n^2}} \quad (\operatorname{mod} \Phi_n
     (q)^2),
  \end{equation*}
  which is the main result of \cite{clark-qbin95}.
  
  In order to prove the congruence \eqref{eq:qljunggren:x}, we now consider
  those summands in \eqref{eq:qchu} which are divisible by $\Phi_n (q)^2$ but
  not divisible by $\Phi_n (q)^3$. These correspond to all but two of the
  $c_i$ taking values $0$ or $n$. More precisely, such a summand is determined
  by indices $1 \leq j_1 < j_2 < \ldots < j_b < j_{b + 1} \leq a$,
  two subindices $1 \leq k < \ell \leq b + 1$, and $1 \leq d
  \leq n - 1$ such that
  \begin{equation*}
    c_i = \left\{ \begin{array}{ll}
       d, & \text{for $i = j_k$,}\\
       n - d, & \text{for $i = j_{\ell}$,}\\
       n, & \text{for $i \in \{j_1, \ldots, j_{b + 1} \}\backslash\{j_k,
       j_{\ell} \}$,}\\
       0, & \text{for $i \not\in \{j_1, \ldots, j_{b + 1} \}$.}
     \end{array} \right.
  \end{equation*}
  For each fixed choice of the $j_i$ and $k, \ell$, the contribution of the
  corresponding summands is
  \begin{equation}
    \sum_{d = 1}^{n - 1} q^{n \sum_{1 \leq i \leq a} (i - 1) c_i
     - \sum_{1 \leq i < j \leq a} c_i c_j} \binom{n}{d}_q
    \binom{n}{n - d}_q . \label{eq:qsum:d}
  \end{equation}
  Note that $q^n \equiv 1$ modulo $\Phi_n (q)$, and
  \begin{equation*}
    \sum_{1 \leq i < j \leq a} c_i c_j \equiv d (n - d) \equiv -
     d^2 \quad (\operatorname{mod} n) .
  \end{equation*}
  Since each binomial in \eqref{eq:qsum:d} is divisible by $\Phi_n (q)$, we
  conclude that \eqref{eq:qsum:d} reduces modulo $\Phi_n (q)^3$ to
  \begin{equation*}
    \sum_{d = 1}^{n - 1} q^{d^2} \binom{n}{d}_q^2 = \binom{2 n}{n}_q -
     [2]_{q^{n^2}},
  \end{equation*}
  where the equality follows from the special case $a = 2$ and $b = 1$ of
  \eqref{eq:qchu}. Since this does not depend on the value of $d$, and because
  there are
  \begin{equation*}
    \binom{a}{b + 1} \binom{b + 1}{2} = \frac{(a - b) b}{2} \binom{a}{b}
  \end{equation*}
  choices for the $j_i$ and $k, \ell$, we find that
  \begin{equation*}
    \binom{a n}{b n}_q \equiv \binom{a}{b}_{q^{n^2}} + \frac{(a - b) b}{2}
     \binom{a}{b} \left[ \binom{2 n}{n}_q - [2]_{q^{n^2}} \right] \quad
     (\operatorname{mod} \Phi_n (q)^3) .
  \end{equation*}
  The general result therefore follows from the special case $a = 2$, $b = 1$,
  which is proved in Lemma~\ref{lem:qwol:x}.
\end{proof}

\begin{acknowledgements}
I am extremely thankful to Ofir Gorodetsky for
motivating me to finally complete and write up the present paper, and for
sharing a preprint of his work \cite{gorodetsky-cong-q} in which he obtains
$q$-congruences, from which he is able to derive the stronger congruences $A
(p^r m) \equiv A (p^{r - 1} m)$ modulo $p^{3 r}$ for the Ap\'ery numbers due
to Beukers and Coster. I am also grateful to Ofir Gorodetsky and Wadim Zudilin
for helpful comments on an early draft of this paper.
\end{acknowledgements}

\end{document}